\documentclass[11pt,oneside]{amsart}
\usepackage{amsmath, ifthen, amsfonts, amssymb, srcltx, amsopn, xcolor, enumerate, comment} 
\usepackage[linktocpage=true]{hyperref}
\usepackage[all]{xy}
\usepackage{pb-diagram, pb-xy}
\usepackage{overpic}
\usepackage{tikzsymbols}
\usepackage{tikz-cd}

\usepackage[T1]{fontenc}

\usepackage[parfill]{parskip}    
\usepackage{fullpage}
\usepackage{graphicx}
\usepackage{amssymb,amsmath,latexsym}
\usepackage{tikz}
\usetikzlibrary{arrows,decorations.markings,decorations.pathreplacing,calc,matrix,intersections}
\tikzset{->-/.style={decoration={markings,mark=at position #1 with {\arrow{>}}},postaction={decorate}}}

\usepackage{epstopdf}
\usepackage{ifpdf}
\usepackage{color}
\usepackage{float}
\usepackage{amscd,amsmath, mathabx}
\usepackage{amssymb,amsmath,latexsym,color,enumerate,tikz}
\usepackage[all]{xypic}       
\usepackage[varg]{pxfonts}
\usepackage{amsmath,amsthm,amsfonts,amssymb,fancyhdr,graphics,relsize,tikz-cd,mathtools,faktor,tikz,changepage}
\usepackage{graphicx}
\usepackage{placeins}

\usepackage{dsfont}

\newcommand{\showcommentsbox}{yes}

\newsavebox{\commentbox}
{\ifthenelse{\equal{\showcommentsbox}{yes}}%
{\footnotemark
        \begin{lrbox}{\commentbox}
        \begin{minipage}[t]{1.25in}\raggedright\sffamily\tiny
        \footnotemark[\arabic{footnote}]}
{\begin{lrbox}{\commentbox}}}%
{\ifthenelse{\equal{\showcommentsbox}{yes}}%
{\end{minipage}\end{lrbox}\marginpar{\usebox{\commentbox}}}
{\end{lrbox}}}

\definecolor{red}{rgb}{1,0,0} 

 \definecolor{darkgreen}{rgb}{0, .7, 0}

 \definecolor{purple}{rgb}{.7, 0, 1}

\tikzset{mynode/.style={draw,circle,fill=black,inner sep=2pt,outer sep=0.5pt}}

\newtheorem*{thm:main}{Theorem \ref{thm:main}}
\newtheorem*{thm:main2}{Theorem \ref{thm:main2}}

\newtheorem{theorem}{Theorem}[section]
\newtheorem*{theorem*}{Theorem}
\newtheorem*{lemma*}{Lemma}
\newtheorem*{corollary*}{Corollary}
\newtheorem{proposition}[theorem]{Proposition}
\newtheorem{lemma}[theorem]{Lemma}
\newtheorem{claim}[theorem]{Claim}
\newtheorem{corollary}[theorem]{Corollary}

\theoremstyle{definition}
\newtheorem{definition}[theorem]{Definition}

\newtheorem{notation}[theorem]{Notation}

\newtheorem{remark}[theorem]{Remark}

\usepackage{hyperref}
 
\begin{document}

\title{Finitely presented subgroups of direct products of graphs of groups with free abelian vertex groups}

\author[M. Casals-Ruiz]{Montserrat Casals-Ruiz}
\address{Ikerbasque - Basque Foundation for Science and Matematika Saila,  UPV/EHU,  Sarriena s/n, 48940, Leioa, Bizkaia, Spain}
\email{montsecasals@gmail.com}

\author[J. Lopez de Gamiz Zearra]{Jone Lopez de Gamiz Zearra}
\address{Stevenson Center 1415, Vanderbilt University, 1326 Stevenson Center Ln, 37212, Nashville, Tennessee, USA}
\email{jone.lopez.de.gamiz.zearra@vanderbilt.edu}

\begin{abstract}
A result by Bridson, Howie, Miller, and Short states that if $S$ is a finitely presented subgroup of the direct product of free groups, then $S$ is virtually a nilpotent extension of a direct product of free groups. Moreover, if $S$ is a subgroup of type $FP_n$ of the direct product of $n$ free groups, then the nilpotent extension is finite, so $S$ is actually virtually the direct product of free groups.

In this paper, these results are generalized to $2$-dimensional coherent right-angled Artin groups. More precisely, we show that a finitely presented subgroup of the direct product of $2$-dimensional coherent RAAGs is still virtually a nilpotent extension of a direct product of subgroups. If $S$ is moreover a type $FP_n$ subgroup of the direct product of $n$ $2$-dimensional coherent RAAGs, then $S$ is commensurable to a kernel of a character of a direct product of subgroups.

Finally, we show that the multiple conjugacy problem and the membership problem are decidable for finitely presented subgroups of direct products of $2$-dimensional coherent RAAGs.
\end{abstract}

\maketitle

\section{Introduction}

In 1984 Baumslag and Roseblade characterized finitely presented subgroups of the direct product of two finitely generated free groups, showing that up to finite index, they are themselves a direct product of free groups. This result was generalized in a series of papers by Bridson, Howie, Miller, and Short, culminating in a characterization of subgroups of direct products of (limit groups over) free groups, assuming that the subgroups satisfy suitable finiteness properties. One of the main consequences of these structural results is that the main algorithmic problems are decidable for finitely presented subgroups of direct products of free groups. 

\emph{Right-angled Artin groups (RAAGs)} are defined by presentations where the relations are commutation of some pairs of generators and so the class of RAAGs extends the class of (direct products of) finitely generated free groups. In view of the previous results, one may wonder if finitely presented subgroups of RAAGs have a tame structure and, in particular, if the main algorithmic problems are decidable in that class. Unfortunately, this is not the case as Bridson showed in \cite{Bridson} that there is a right-angled Artin group $A$ and a finitely presented subgroup $S < A\times A$ for which the conjugacy and the membership problems are undecidable.

This work is part of a series that aims to describe the structure of finitely presented subgroups of the direct product of (limit groups over) coherent RAAGs. This programme was carried over in \cite{LopezDeGamiz} for the subclass of RAAGs whose finitely generated subgroups are again RAAGs, called \emph{Droms RAAGs}. Furthermore, in \cite{CasalsLopezdeGamiz} we began the study for the class of \emph{$2$-dimensional coherent RAAGs}. More precisely, we generalized Baumslag and Roseblade's result for free groups and we described the structure of finitely presented subgroups of the direct product of two $2$-dimensional coherent RAAGs: they are virtually abelian extensions of direct products.

The next natural step in this programme is to study subgroups of direct products of finitely many $2$-dimensional coherent RAAGs, thus extending the binary case to the finitary one. This paper addresses this problem and we prove the following:

\begin{theorem*}
Let $G_i$ be a $2$-dimensional coherent RAAG where $i\in \{1,\dots,n\}$, define $G$ to be $G_1\times \cdots \times G_n$ and let $S$ be a finitely presented full subdirect product of $G$. Define $L$ to be $L_1\times \cdots \times L_n$, where $L_i= S \cap G_i$.

Then $L < S < G$ and $G\slash L$ is virtually nilpotent. In particular, there is a subnormal series $$S_0=M_0 \lhd M_1 \lhd \dots \lhd M_{k-1} \lhd M_k$$ where $M_i\slash M_{i-1}$ is abelian, $S_0$ has finite index in $S$ and $M_k$ has finite index in $G$.
\end{theorem*}

Furthermore, we show that these finitely presented subgroups have good algorithmic behavior. Namely, we show the following:

\begin{corollary*}
Finitely presented full subdirect products of the direct product of $2$-dimensional coherent RAAGs have decidable multiple conjugacy and membership problems.   
\end{corollary*}

This corollary shows that Bridson's example of a right-angled Artin group $A$ and an algorithmically bad finitely presented subgroup of $A\times A$ is not $2$-dimensional coherent (and we conjecture that $A$ cannot be coherent).

In fact, our results apply to a wider class of groups, the class $\mathcal{G}$, which is the class of cyclic subgroup separable graphs of groups with free abelian vertex groups and cyclic edge groups that act faithfully on the associated Bass-Serre tree. This class contains, among others, $2$-dimensional coherent RAAGs and residually finite tubular groups with trivial center. Recall that a \emph{tubular group} is a finitely generated graph of groups with $\mathbb{Z}^2$ vertex groups and $\mathbb{Z}$ edge groups.

\begin{thm:main}
Let $G\coloneqq G_1 \times \dots \times G_n$ where $G_i\in \mathcal G$ is finitely generated and let $S<G$ be a finitely presented full subdirect product. Define $L$ to be $L_1\times \cdots \times L_n$, where $L_i= S \cap G_i$.

Then $L < S < G$ and $G\slash L$ is virtually nilpotent. In particular, there is a subnormal series $$S_0=M_0 \lhd M_1 \lhd \dots \lhd M_{k-1} \lhd M_k$$ where $M_i\slash M_{i-1}$ is abelian, $S_0$ has finite index in $S$ and $M_k$ has finite index in $G$.    
\end{thm:main}

We also consider subgroups with stronger finiteness conditions, namely type $FP_n$ subgroups, and show the following:

\begin{thm:main2}
Let $G\coloneqq G_1 \times \dots \times G_n$ where $G_i\in \mathcal G$ is finitely generated and let $S<G$ be a full subdirect product of type $FP_n$.

There there is $S_0$ a finite index subgroup of $S$ and $G_0$ a finite index subgroup of $G$ such that $S_0$ is normal in $G_0$ and $G_0 \slash S_0$ is free abelian.
\end{thm:main2}

The free abelian factor in the description of $S$ is directly related to the edge groups of the decomposition of the groups in our class. In particular, if we consider graphs of groups in $\mathcal G$ that admit a non-trivial free product decomposition, we deduce the following theorem and recover the result of Baumslag-Roseblade and Bridson-Howie-Miller-Short for direct products of free groups:

\begin{theorem*}
Let $\mathcal{G}^\prime$ be the subclass of $\mathcal{G}$ containing the groups which have a non-trivial free product decomposition and let $S$ be a subgroup of type $FP_n$ of the direct product of $n$ groups in the class $\mathcal{G}^\prime$. Then, $S$ is virtually the direct product of groups in $\mathcal{G}^\prime$.
\end{theorem*}

It is worth mentioning the difficulties one encounters when extending the results from direct products of free groups to direct products of $2$-dimensional coherent RAAGs. They mainly come from the fact that, while free groups do not fiber, coherent RAAGs do and so they have nontrivial finitely generated normal subgroups that are not necessarily of finite index (although the quotient by the normal subgroup is always virtually abelian, see \cite{CasalsLopezdeGamiz2}). Moreover, the dimension of the coherent RAAG controls the dimension of the abelian group over which they may fiber. In particular, $2$-dimensional coherent RAAGs can only fiber over $\mathbb Z$ and this is the reason for the restriction on the dimension of the coherent RAAG. Indeed, when studying subgroups of the direct product of two coherent RAAGs in \cite{CasalsLopezdeGamiz}, we first show that any finitely presented subgroup is an extension of a direct product by a $\mathbb{Z}^m$-by-$\mathbb{Z}^m$ group, where $m$ is bounded by the dimension of the coherent RAAG. In the $2$-dimensional case, we then prove that this $\mathbb{Z}$-by-$\mathbb{Z}$ group is a quotient of a Baumslag-Solitar group, and using the structural theory of Baumslag-Solitar groups, we conclude that the $\mathbb{Z}$-by-$\mathbb{Z}$ group is in fact free abelian. In the current general setting of the direct product of $n$ $2$-dimensional coherent RAAGs, the proof involves studying groups of the form nilpotent-by-$\mathbb{Z}^n$ and one needs to show that the nilpotent subgroup is actually central.

The paper is organized as follows. In Section \ref{Section1}, we study general subdirect products of groups. In particular, we describe some conditions that determine that finitely presented full subdirect products are virtually nilpotent extensions of direct products of subgroups.

In Section \ref{Section2}, we restrict to the study of finitely presented subgroups of $2$-dimensional coherent RAAGs and more generally, of groups in the class $\mathcal G$. We prove that the conditions of the previous section hold for this class of groups and combining these results, we describe the structure of the finitely presented subgroups in Theorem \ref{thm:main}. 

In Section \ref{Section3}, we consider subgroups of type $FP_n$ and describe their structure in Theorem \ref{thm:main2}.

Finally, in Section \ref{Section4} we address the algorithmic problems and show that the multiple conjugacy and membership problems are decidable for finitely presented subgroups of direct products of groups in the class.

\section{Structure of subdirect products of groups}
\label{Section1}

In this section, we investigate the structure of full subdirect products of general groups. We define some normal subgroups using the projection maps into some of the factors and the associated kernels and determine sufficient conditions on these subgroups that imply a good structure for the finitely presented full subdirect products, namely they are virtually a nilpotent extensions of direct products of subgroups, see Theorem \ref{thm:fp_subdirect_nilpotent quotient}.

\begin{definition}[Full subdirect product]
A subgroup $H < A_1 \times \cdots \times A_n$ is called \emph{full} if $H$ intersects nontrivially each factor, i.e. $H \cap A_i \neq 1$ for all $i\in \{1, \dots, n\}$, and it is a \emph{subdirect product} if the natural epimorphism from $A_1 \times \cdots \times A_n$ to $A_i$ restricts to an epimorphism from $H$ to $A_i$ for each $i\in \{1, \dots, n\}$.   
\end{definition}

From now on, we will consider that $S < G_1\times \dots \times G_n$ is a full subdirect product.

\begin{notation}[Distinguished normal subgroups]\label{notation}

Let $I=\{1, \dots, n\}$.

For each (non-empty) subset $J \subset I$, we define $p_J$ to be the natural projection homomorphism

\[ p_J \colon S \to \prod\limits_{j\in J} G_j.\]

In addition, we denote by $K_i$ the kernel of the epimorphism $p_{\{i\}}\colon S \to G_i$ and by $L_i$ the subgroup $S\cap G_i < G_i$ for $i\in I$. 

If $J$ is a subset of $I \setminus \{i\}$, we denote by $N_J^{(i)}$ the subgroup $p_i(\ker(p_J)) < G_i$. Notice that $N_J^{(i)}$ is normal in $G_i$ since $p_{\{i\}}\colon S\to G_i$ is surjective and $\ker(p_J) \triangleleft S$. By convention, if $J$ equals $I$, then we set $N_J^{(i)}$ to be $\{1\}$.

For $k\in \{1, \dots, n-1\}$ let $b(k)$ be the number $\binom{n-1}{k}$ and let $J_1, \dots, J_{b(k)}$ be the subsets of $I\setminus \{i\}$ of size $k$. The \emph{$k$-union} $U(k)^{(i)}$ stands for the normal subgroup \[U(k)^{(i)}=N_{J_1}^{(i)}\cdots N_{J_{b(k)}}^{(i)}\] of $G_i$, whereas the \emph{$k$-intersection} $I(k)^{(i)}$ is $N_{J_1}^{(i)} \cap \ldots \cap N_{J_{b(k)}}^{(i)}$.

Similarly, for $J\subseteq I\setminus \{i\}$ we define $U_J(1)^{(i)}$ to be the subgroup $\prod_{l\in J} N_{\{l\}}^{(i)}$ and $I_J(|J|-1)^{(i)}$ to be the subgroup $\bigcap_{l\in J} N_{J\setminus \{l\}}^{(i)}$.
\end{notation}

We record the following basic observations about the subgroups $N_J^{(i)}$, $U(k), I(k)$ and $L_i$.

\begin{remark}[Basic properties of the distinguished normal subgroups]\label{rem:properties_N_J}\

    \begin{enumerate}[(i)]
        \item By definition, $\ker(p_J)$ equals $\{(g_1, \dots, g_n) \in S \mid g_j = 1 \text{ for all } j\in J\}$;
        \item if $J=I \setminus \{i\}$, then $N_J^{(i)}=L_i$;
        \item for all $J^\prime \subset J\subset I \setminus \{i\}$ we have that $N_{J}^{(i)} < N_{J^\prime}^{(i)}$. Moreover, $N_J^{(i)} < U(|J|)^{(i)}$ and $I(|J|)^{(i)} < N_J^{(i)}$.
        \item for all $J,J^\prime \subseteq I \setminus \{i\}$ we get that the commutator $[N_J^{(i)},N_{J^\prime}^{(i)}]$ lies in $N_{J \cup J'}^{(i)}$. This is a consequence of the fact that $[\ker(p_J), \ker(p_{J^\prime})]< \ker(p_{J\cup J^\prime})$. 
    \end{enumerate}
\end{remark}

\begin{lemma}[Relations between the distinguished normal subgroups]\label{lem:central}
In the above notation, $[U_J(1)^{(i)}, I_J(|J|-1)^{(i)}]< N_J^{(i)}.$
\end{lemma}

\begin{proof}
For each $l\in J$, since $I_J(|J|-1)^{(i)}$ is a subgroup of $ N_{J \setminus \{l\}}^{(i)}$, from Remark \ref{rem:properties_N_J} (iv) we get that \[[N_{\{l\}}^{(i)}, I_J(|J|-1)^{(i)}]< N_J^{(i)},\] so the result follows from the definition of $U_J(1)^{(i)}$.
\end{proof}

\begin{theorem}[Structure of finitely presented full subdirect products of groups]\label{thm:fp_subdirect_nilpotent quotient}
Let $S < G_1 \times \dots \times G_n$ be a full subdirect product.
Assume that $G_i$ satisfies that
\begin{itemize}
    \item[(1)] $G_i\slash N_{\{k\}}^{(i)}$ is virtually abelian for each $k\in I\setminus \{i\}$;
    \item[(2)] for each $J\subset I\setminus \{i\}$ and for each finite index subgroup $K<U_J(1)^{(i)}$, there is a finite index subgroup $H_K$ of $G_i$ such that $K < H_K$ and $[H_K,H_K]=[K,K]$.
\end{itemize}
Then $G_i \slash L_i$ is virtually nilpotent.
\end{theorem}

\begin{proof}

The proof is by induction on $|J|$ and it essentially uses Hall's criterion for nilpotence: a group $G$ is nilpotent whenever it has a normal subgroup $N$ such that $G\slash [N,N]$ and $N$ are nilpotent.

More precisely, we prove the following by induction on $|J|$:

\begin{claim}\label{claim:virnilpotent}
If $J \subseteq I \setminus \{i\},$ then $G_i \slash N_J^{(i)}$ is virtually nilpotent.
\end{claim}

\begin{proof}[Proof of Claim \ref{claim:virnilpotent}]

If $|J|=1$, then $G_i\slash N_J^{(i)}$ is virtually abelian by assumption, so in particular, it is virtually nilpotent. 

Now assume that $|J|=k$ and that the statement holds for subsets of $I\setminus \{i\}$ with less than $k$ elements. For each element $l$ in $I$ such that $l\notin J$, by induction on the size of the set $J\setminus \{l\}$ we have that $G_i \slash N_{J\setminus \{l\}}^{(i)}$ is virtually nilpotent, so there is $H_{J\setminus \{l\}}$ a finite index subgroup in $G_i$ such that $H_{J\setminus \{l\}} \slash N_{J\setminus \{l\}}^{(i)}$ is nilpotent, say of nilpotency class $c_l$. Let $M_J$ be the intersection $\bigcap_{l\notin J} H_{J\setminus \{l\}}$ and let $C$ be $\max \{ c_l \mid l\notin J\}$. Then $M_J$ has finite index in $G_i$ and the $C$-th term of the lower central series of $M_J$ is a subgroup of $N_{J\setminus \{l\}}^{(i)}$ for each $l$. Hence, its $C$-th term of the lower central series is a subgroup of their intersection, and therefore the group 
\[ \frac{M_J}{I_J(|J|-1)^{(i)}}\]
is nilpotent. Notice that if we define $K_J$ to be $M_J \cap U_J(1)^{(i)}$, then $K_J$ is a finite index subgroup of $U_J(1)^{(i)}$ and the subgroup $K_J \slash I_J(|J|-1)^{(i)}$ of $G_i\slash I_J(|J|-1)^{(i)}$ is nilpotent. 

Furthermore, $I_J(|J|-1)^{(i)}\slash N_J^{(i)}$ lies in the center of the group $U_J(1)^{(i)}\slash N_j^{(i)}$ (see Lemma \ref{lem:central}), so in particular, in the center of $K_J\slash N_J^{(i)}$. Hence, $K_J\slash N_J^{(i)}$ is nilpotent.

By assumption, there exists a finite index subgroup $H_{K_J}<G_i$ such that $K_J < H_{K_J}$ and $[H_{K_J},H_{K_J}]=[K_J,K_J]$. Since \[[H_{K_J},H_{K_J}]=[K_J,K_J]<K_J<H_{K_J}\] and $H_{K_J}\slash [H_{K_J},H_{K_J}]$ is abelian, we have that $K_J\slash [H_{K_J},H_{K_J}]$ is a normal subgroup of $H_{K_J}/[H_{K_J},H_{K_J}]$ and so $K_J$ is normal in $H_{K_J}$.

In conclusion, $K_J\slash N_J^{(i)}$ is a nilpotent normal subgroup of $H_{K_J}\slash N_J^{(i)}$ and 
\[(H_{K_J}\slash N_J^{(i)}) \slash [K_J \slash N_J^{(i)}, K_J\slash N_J^{(i)}] \cong H_{K_J}\slash N_J^{(i)}[K_J,K_J]=H_{K_J}\slash N_J^{(i)}[H_{K_J},H_{K_J}]\] 
is abelian. Hence, from Hall's criterion we conclude that $H_{K_J} \slash N_J^{(i)}$ is nilpotent and so $G_i\slash N_J^{(i)}$ is virtually nilpotent.
\end{proof}

In particular, if we take $J$ to be $I\setminus \{i\}$, then $N_J^{(i)}=L_i$ and so we conclude that $G_i\slash L_i$ is virtually nilpotent.

\end{proof}

\begin{lemma}[Virtually nilpotent quotients]\label{lem:justfg}
Let $S< G_1\times \cdots \times G_n$ be a full subdirect product and take $J\subseteq I\setminus \{i\}$. For each element $k\in J$, let $J_k = (J\setminus \{k\}) \cup \{i\}$. If $G_k \slash N_{J_k}^{(k)}$ is virtually nilpotent for each $k\in J$, then $G_i \slash N_J^{(i)}$ is virtually nilpotent.  
\end{lemma}

\begin{proof}
For each $k\in J$ take $H_k$ to be a finite index subgroup of $G_k$ such that $N_{J_k}^{(k)} < H_k$ and $H_k \slash N_{J_k}^{(k)}$ is nilpotent, say of nilpotency class $n_k$. Moreover, define $n_0$ to be $\max \{n_k \mid k\in J \}$.

Since $\bigoplus_{k\in J} H_k$ has finite index in $\bigoplus_{k\in J} G_k$, then defining $S_0$ to be \[p_J^{-1}\Big{(} \bigoplus_{k\in J} H_k \cap p_J(S)\Big{)},\] the subgroup $S_0$ has finite index in $S$, so $H\coloneqq  p_{\{i\}}(S_0)$ has finite index in $G_i$. The goal is to show that the $n_0$-term of the lower central series of $H$ lies in $N_J^{(i)}$.

Let $x_1,\dots, x_{n_0}\in H$. By the definition of $H$, for each $j\in \{1,\dots,n_0\}$ there is an element in $S$ of the form $(y_1^j,\dots, y_{i-1}^j,x_j,y_{i+1}^j,\dots, y_n^j)$ with the extra assumption that $y_k^j \in H_k$ if $k\in J.$ The $n_0$-fold
\[ [(y_1^1,\dots, y_{i-1}^1,x_1,y_{i+1}^1,\dots, y_n^1),\dots, (y_1^{n_0},\dots, y_{i-1}^{n_0},x_{n_0},y_{i+1}^{n_0},\dots, y_n^{n_0})]\]
equals
\[ ([y_1^1,\dots, y_1^{n_0}],\dots, [y_{i-1}^1,\dots, y_{i-1}^{n_0}], [x_1,\dots,x_{n_0}], [y_{i+1}^1,\dots, y_{i+1}^{n_0}], \dots, [y_{n}^1,\dots, y_{n}^{n_0}])\]
and by assumption $[y_j^1,\dots, y_j^{n_0}]$ belongs to $N_{J_k}^{(k)}$ if $k\in J$. Therefore, it follows that $[x_1,\dots, x_{n_0}]$ lies in $N_J^{(i)}$.
\end{proof}

\section{The class of cyclic subgroup separable graphs of groups with free
abelian vertex groups and cyclic edge groups}
\label{Section2}

Let $\mathcal{G}$ be the class of cyclic subgroup separable graphs of groups with free abelian vertex groups and infinite cyclic or trivial edge groups that act faithfully on the associated Bass-Serre tree. In \cite{CasalsLopezdeGamiz} we introduced this class of groups and we studied some properties of the actions.

The goal of this section is to review these results and to further explore finitely presented full subdirect products of groups in $\mathcal{G}$. More concretely,  we aim to prove that if $G_1,\dots, G_n$ are finitely generated groups in $\mathcal{G}$ and if $S$ is a finitely presented full subdirect product of $G_1\times \cdots \times G_n$, then $G_i \slash L_i$ is virtually nilpotent (see Theorem \ref{thm:main}). The strategy is to show that in this setting, the conditions (1) and (2) from Theorem \ref{thm:fp_subdirect_nilpotent quotient} are satisfied.

Let $G$ be a group in the class $\mathcal{G}$. Any splitting of $G$ as a graph of groups with free abelian vertex groups and cyclic edge groups is called a \emph{standard splitting} of $G$.

\begin{lemma}\label{Lemma 1}\cite[Lemma 3.3]{CasalsLopezdeGamiz}
Let $G$ be a finitely generated group in $\mathcal{G}$ and let $T$ be the Bass-Serre tree corresponding to a standard splitting of $G$. Suppose that $N$ is a non-trivial normal subgroup of $G$. Then, $N$ contains hyperbolic elements and it acts minimally on $T$.
\end{lemma}

\begin{proposition}\label{Proposition 0}\cite[Proposition 3.4]{CasalsLopezdeGamiz}
Let $G$ be a finitely generated group in $\mathcal{G}$ and let $T$ be the Bass-Serre tree corresponding to a standard splitting of $G$. Suppose that $N$ is a non-trivial finitely generated normal subgroup of $G$. Then, $G \slash N$ is virtually cyclic. 

Furthermore, $N$ is a free product of free abelian groups whose ranks are bounded above by $r$, where $r$ is the maximum of the ranks of the (free abelian) vertex groups of $G$. In particular, if $G$ is a residually finite tubular group, $N$ is a finitely generated non-trivial normal subgroup of $G$ and $G\slash N$ is virtually $\mathbb Z$, we have that $N$ is a free group.
\end{proposition}

We next show some other properties that we will need when working with finitely presented full subdirect products of groups in $\mathcal{G}$.

\begin{lemma}[Finitely generated normal subgroups]\label{lemma:edgegroups}
Let $G$ be a finitely generated group in $\mathcal{G}$ and let $N$ be a normal subgroup of $G$ such that $G \slash N$ is virtually $\mathbb{Z}$. Suppose that for each edge group $\langle e \rangle$ of $G$ in the standard splitting of $G$, we have that $N \cap \langle e \rangle =1$. Then $N$ is finitely generated.
\end{lemma}

\begin{proof}
Let $T$ be the Bass-Serre tree corresponding to the splitting of $G$. Since $N$ is a subgroup of $G$, it also acts on $T$. If we check that $T \slash N$ is finite, then $N$ is finitely generated because the vertex groups are finitely generated abelian.

Hence, it suffices to show that the number of edges in $T \slash N$ is finite. That is, it is enough to show that \[ \mathlarger{\sum} |N \backslash G \slash \langle e \rangle| < \infty,\]
where the sum is taken over the edge groups in the standard splitting of $G$.

For each edge group $\langle e \rangle$ in the splitting of $G$, as $N$ is normal in $G$ we have that \[ | N \backslash G \slash \langle e \rangle|= | G \slash N\langle e \rangle |.\]
Note that $N \langle e \rangle$ has finite index in $G$ because by assumption $G \slash N$ is virtually $\mathbb{Z}$ and $N \cap \langle e \rangle=1$. Hence, it follows that $T\slash N$ is finite and $N$ is finitely generated.
\end{proof}

\begin{lemma}[Quotients by normal subgroups with finitely generated cyclic extensions]\label{lem:G/L is virtually abelian}
Let $G$ be a finitely generated group in $\mathcal{G}$ and let $T$ be the Bass-Serre tree corresponding to a standard splitting of $G$.

Let $N$ and $H$ be subgroups of $G$ such that $N < H< G$, $N$ is normal in $G$ (not necessarily finitely generated) and $H$ is finitely generated and $H\slash N$ is virtually cyclic. 

Then $G\slash N$ is virtually free abelian (of rank less than or equal to 2). Moreover, in the case where $G \slash N$ is virtually $\mathbb{Z}^2$, the group $H$ is a finitely generated free group.
\end{lemma}

\begin{proof}
Since $H \slash N$ is by assumption virtually cyclic, there exists a finite index subgroup $H^\prime<H$ and an element $y\in H^\prime$ such that $H^\prime \slash N =\langle y\rangle$, i.e. $H^\prime=\langle N, y\rangle$.

The proof follows from results in \cite{CasalsLopezdeGamiz} replacing the notation of the normal subgroup $L_2$ by $N$. More precisely, in \cite[Lemma 4.6]{CasalsLopezdeGamiz} we show that for each edge group $\langle e \rangle$ of $G$, the group $G$ is a finite union of double cosets of the subgroups $\langle N,y\rangle$ and $\langle e \rangle$. After that, in \cite[Lemma 4.7]{CasalsLopezdeGamiz}, we prove that if there is a vertex stabilizer $G_v$ in $T$ such that $G_v N\slash N$ is not virtually cyclic, then $G \slash N$ is virtually abelian (with free rank bounded by 2). In \cite[Lemma 4.8]{CasalsLopezdeGamiz}, we analyze the other alternative and show that if for each vertex stabilizer $G_v$ in $T$ we have that $G_vN\slash N$ is virtually cyclic, then $G \slash N$ is virtually abelian (with free rank bounded by 1). We also show that in the case where $G \slash N$ is virtually $\mathbb{Z}^2$, then the group $H$ is free.
\end{proof}

The following result is not required to show that the hypotheses of Theorem \ref{thm:fp_subdirect_nilpotent quotient} hold, but we find the statement quite surprising.

\begin{lemma}[Quotients by finitely generated normal subgroups]\label{lem:fg_normal}
Let $G \coloneqq G_1 \times \dots \times G_n$ where $G_i\in \mathcal G$ is finitely generated. If $N < G$ is a finitely generated full normal subgroup, then $G\slash N$ is virtually abelian.
\end{lemma}

\begin{proof}
Let us denote by $\pi_i \colon G \mapsto G_i$ the natural projection homomorphism for each $i\in \{1,\dots,n\}$. Then, recalling that we denote $N\cap G_i$ by $L_i$ for $i\in \{1,\dots,n\}$, we have a chain of groups
\[ L_1 \times \cdots \times L_n < N < \pi_1(N)\times \cdots \times \pi_n(N)< G.\]
The aim is to show that for any $i\in \{1,\dots,n\}$ the quotient group $G_i \slash L_i$ is virtually abelian. We do this for $i=1$ as the other cases are identical.

Since $N$ is a normal subgroup of $G$, we view $N$ as the kernel of a homomorphism, that is, there is a group $Q$ and a homomorphism $\phi \colon G \to Q$ such that $\ker \phi= N$.

Firstly, note that the subgroup $\pi_1(N)$ is non-trivial, finitely generated and normal in $G_1$, so from Proposition \ref{Proposition 0} we deduce that $G_1 \slash \pi_1(N)$ is virtually cyclic, so there is $H_1$ a finite index subgroup in $G_1$ such that $H_1 \slash \pi_1(N)$ is cyclic.

Secondly, the group $\pi_1(N)\slash L_1$ lies in the center of the group $G_1 \slash L_1$. Indeed, let us take $x\in \pi_1(N)$ and $y\in G_1$. Note that showing that $x$ and $y$ commute in $G_1 \slash L_1$ is equivalent to showing that $\phi(x)\phi(y)= \phi(y)\phi(x)$. Since $x$ is an element of $\pi_1(N)$, there is an element of the form $(x,y_2,\dots,y_n)$ in $N$. Thus, $\phi(x)= \phi((y_2,\dots,y_n))$. As a consequence,
\[ \phi(x)\phi(y)=\phi((y_2,\dots,y_n))\phi(y)=\phi(y)\phi((y_2,\dots, y_n))=\phi(y)\phi(x),\]
where the second equality holds because $y$ and $(y_2,\dots,y_n)$ commute. To sum up, we have showed that the group $\pi_1(N) \slash L_1$ is central in the group $G_1\slash L_1$, so in particular it is abelian and central in $H_1 \slash L_1$. 

As a consequence, we have a central short exact sequence
\[ 1 \to \pi_1(N) \slash L_1 \to H_1 \slash L_1 \to H_1 \slash \pi_1(N) \to 1,\]
and $H_1 \slash \pi_1(N)$ is either finite or infinite cyclic. In the former case, $H_1 \slash L_1$ is virtually abelian, and in the latter case, we get that $H_1 \slash L_1$ is the semidirect product of $\pi_1(N)\slash L_1$ and an infinite cylic group with $\pi_1(N)\slash L_1$ central, so it is also virtually abelian.
\end{proof}

\begin{remark}
The previous lemma shows that under the above conditions, $G\slash N$ is finite-by-abelian. Nevertheless, in contrast to what happens in the case of right-angled Artin groups (see \cite[Theorem 2.5]{CasalsLopezdeGamiz2}), $G\slash N$ may not be abelian-by-finite.

For instance, let $G=\langle a, t \mid a^t = a^{-1}\rangle$ and $N= \langle t^2 \rangle$. Then, $G \slash N$ is the infinite dihedral group $D_\infty$, so it is isomorphic to $\mathbb{Z} \rtimes \mathbb{Z}_2$.
\end{remark}

Coming back to our goal, we first aim to prove condition (1) from Theorem \ref{thm:fp_subdirect_nilpotent quotient}, i.e. that the quotient $G_i\slash N_{\{k\}}^{(i)}$ is virtually abelian.

\begin{lemma}[Finite generation of cyclic extensions of kernels]\label{claim:Ki}
Let $G\coloneqq G_1 \times \dots \times G_n$ where $G_i\in \mathcal G$ is finitely generated. Let $S<G$ be a finitely presented full subdirect product. Then, for each cyclic (or trivial) edge group $\langle e \rangle$ of $G_i$ in its standard splitting, there is $s_1\in \langle e \rangle$ and a lift $\hat {s}_1 \in S$ such that $\langle K_i, \hat{s}_1 \rangle$ is finitely generated, where $K_i$ is the kernel of the epimorphism $p_{\{i\}}:S \to G_i$.  
\end{lemma}

\begin{proof}
Let $\langle e \rangle <G_i$ be an edge group of $G_i$ and consider the splitting of $G_i$ with respect to this edge. As $L_i= S \cap G_i$ is non-trivial and normal in $G_i$, from Lemma \ref{Lemma 1} we have that it contains a hyperbolic element, say $t\in L_i$. Since $t$ is hyperbolic with respect to the splitting of $G_i$ over $\langle e\rangle$, its axis contains an edge with corresponding edge group $\langle e ^g \rangle$ for some $g\in G_i$. The subgroup $L_i$ is normal in $G_i$, so up to conjugacy we can assume that $g$ equals $1$.

The group $G_i$ is by definition cyclic subgroup separable, so from \cite[Theorem 3.1]{BridsonHowie} we get that there is a finite index subgroup $M$ of $G_i$ which is an HNN extension with stable letter $t$ and associated subgroup $\langle s_1 \rangle$ for some $s_1\in \langle e\rangle$. As $G_i$ is finitely generated, $M$ is also finitely generated, namely, $M= \langle t, s_1,\dots, s_m \rangle$ with $s_1\in \langle e\rangle$ and $s_2= t^{-1}s_1 t$.

Let us denote by $M^\prime$ the preimage of $M$ in $S$. Then, taking into account that $M$ has finite index in $G_i$, the subgroup $M^\prime$ also has finite index in $S$, and since $S$ is finitely presented, so is $M^\prime$. The HNN decomposition of $M$ induces a decomposition of $M^\prime$ as an HNN extension. Let us pick $\hat{s}_j\in S$ such that $p_{\{i\}}(\hat{s}_j)= s_j$ for $j\in \{1,\dots,m\}$ with the restriction that $\hat{s}_2$ needs to be $t^{-1} \hat{s}_1 t$. Note that, by definition, $t$ is an element in $S$. We then have that
\[ M^\prime= \langle K_i, \hat{s}_1,\dots, \hat{s}_m, t \mid t^{-1}\hat{s}_1 t= \hat{s}_2, t^{-1}bt=b, \forall b\in K_i, \mathcal{R} \rangle,\]
where $\mathcal{R}$ is a set of relations in the elements $K_i \cup \{\hat{s}_1,\dots,\hat{s}_m\}$. Since $M^\prime$ is finitely generated, there are $a_1,\dots,a_k$ in $K_i$ such that
\[ M^\prime= \langle a_1,\dots,a_k, \hat{s}_1,\dots, \hat{s}_m, t \mid t^{-1}\hat{s}_1 t= \hat{s}_2, t^{-1}bt=b, \forall b\in K_i, \mathcal{R} \rangle.
\]
Let $D$ be the subgroup $\langle a_1,\dots,a_k, \hat{s}_1, \dots, \hat{s}_m \rangle$. We next show that $K_i < D$. The subgroup $K_i$ is finitely generated as a normal subgroup of $M^\prime$ because $M$ is finitely presented, and so it is generated by the conjugates $a_i^g$ for $g\in M^\prime$. Note that if $x$ is an element in $K_i$, since $t$ is an element in $L_i$, then $x^{t^n}=x$ for all $n\in \mathbb{Z}$, so in fact, $K_i$ is generated by the conjugates of $a_i$ by elements in $\langle a_1,\dots,a_k, \hat{s}_1,\dots,\hat{s}_m \rangle$.

In conclusion,
\[ M^\prime= \langle D, t \mid t^{-1}\hat{s}_1 t= \hat{s}_2, t^{-1}bt=b, \forall b\in K_i \rangle.
\]
To conclude, $M^\prime$ is finitely generated, so by \cite[Lemma 2]{Miller} we deduce that $\langle K_i, \hat{s}_1 \rangle$ is finitely generated, and $\hat{s}_1$ is a lift of $s_1\in \langle e \rangle$.
\end{proof}

\begin{remark}\label{remark:SES}
Notice that if $K_i$ is not finitely generated, then $H_0( \langle \hat{s}_1 \rangle; H_1(K_i; \mathbb{Q}))$ is finite-dimensional. Indeed, by Lemma \ref{claim:Ki}, there is an element $\hat{s}_1\in S$ such that $\langle K_i, \hat{s}_1 \rangle$ is finitely generated. If $K_i$ is not finitely generated, then the intersection $\langle \hat{s}_1 \rangle \cap K_i$ is trivial. From the short exact sequence
\[ 1 \to K_i \to \langle K_i, \hat{s}_1 \rangle \to \langle \hat{s}_1 \rangle \to 1  \]
we have that the associated Lyndon-Hochschild-Serre spectral sequence is a $2$-column spectral sequence, so $E^\infty$ coincides with $E^2$. Since $H_1(\langle K_i, \hat{s}_1 \rangle; \mathbb{Q})$ is finite-dimensional, then
\[ E^\infty_{0,1}= E^2_{0,1}= H_0( \langle \hat{s}_1 \rangle; H_1(K_i; \mathbb{Q}))\]
is also finite-dimensional.
\end{remark}

\begin{corollary}\label{lem:prop_1_subdirect}
In the Notation \ref{notation}, for each $i\in I$ and each $l\in I\setminus \{i\}$ we have that $G_i\slash N_{\{l\}}^{(i)}$ is virtually abelian (of rank less than or equal to 2).     
\end{corollary}
 
\begin{proof}

Let $i\in I$ and $l\in I \setminus \{i\}$. Let $\langle e \rangle < G_l$ be an edge group. By Lemma \ref{claim:Ki}, there is $s_1\in \langle e\rangle$ and $\hat{s}_1 \in S$ a lift of $s_1$ such that $\langle K_l, \hat{s}_1\rangle < S$ is finitely generated. 

Since $\langle K_l, \hat{s}_1\rangle < S$ is finitely generated, so is $p_{\{i\}}(\langle K_l, \hat{s}_1 \rangle)= \langle N_{\{l\}}^{(i)}, p_{\{i\}}(\hat{s}_1) \rangle < G_i$. 

We have a chain of subgroups \[N_{\{l\}}^{(i)} < \langle N_{\{l\}}^{(i)}, p_{\{i\}}(\hat{s}_1) \rangle < G_i,\] where the subgroup $N_{\{l\}}^{(i)}$ is normal in $G_i$, the subgroup $\langle N_{\{l\}}^{(i)}, p_{\{i\}}(\hat{s}_1) \rangle$ is finitely generated and $\langle N_{\{l\}}^{(i)}, p_{\{i\}}(\hat{s}_1) \rangle \slash N_{\{l\}}^{(i)}$ is cyclic. Thus, from Lemma \ref{lem:G/L is virtually abelian} we conclude that $G_i \slash N_{\{l\}}^{(i)}$ is virtually abelian.
\end{proof}

Showing that condition (2) from Theorem \ref{thm:fp_subdirect_nilpotent quotient} holds will require more work. We start by simplifying the condition by checking that it is enough to consider the groups $U_J(1)^{(i)}$ that are finitely generated.

\begin{lemma}[$U_{J_k}(1)^{(k)}$ are finitely generated]\label{lemma:UJfg}
In the Notation \ref{notation}, we have that $G_i \slash U_J(1)^{(i)}$ is either finite or virtually $\mathbb{Z}$ for each $i\in I$ and each $J\subset I\setminus \{i\}$ with $|J|\ge 2$.

Moreover, if $G_i \slash U_J(1)^{(i)}$ is virtually $\mathbb{Z}$ and $U_J(1)^{(i)}$ is not finitely generated, then the groups $U_{J_k}(1)^{(k)}$ are finitely generated for all $k\in J$, where $J_k= (J\setminus \{k\}) \cup \{i\}$.    
\end{lemma}

\begin{proof}
In order to show the first part of the statement, since for any two elements $l, k$ of $J$ we have the chain $N_{\{k\}}^{(i)} < U_{\{l,k\}}(1)^{(i)} < G_i$ and $G_i \slash N_{\{k\}}^{(i)}$ is virtually abelian of rank at most 2 (see Corollary \ref{lem:prop_1_subdirect}), it suffices to show that $G_i \slash U_{\{l,k\}}(1)^{(i)}$ cannot be virtually $\mathbb{Z}^2$.

Suppose towards contradiction that $G_i \slash U_{\{l,k\}}(1)^{(i)}$ is virtually $\mathbb{Z}^2$. From the above it follows that $N_{\{l\}}^{(i)}$ and $N_{\{k\}}^{(i)}$ have finite index in $ U_{\{l,k\}}(1)^{(i)}$.

By Lemma \ref{claim:Ki}, for each edge group $\langle e_i \rangle$ of $G_i$ there is $s_i\in \langle e_i\rangle$ and a lift $\hat{s}_i$ of the form $(x_1,\dots, x_{i-1}, s_i, x_{i+1},\dots, x_n)\in S$ such that $\langle K_i, \hat{s}_i \rangle$ is finitely generated.

Let us first show that there is $x\in N_{\{l\}}^{(i)}$ such that $[x, e_i]=1$. Let $\langle e_i' \rangle$ be another edge group of $G_i$ such that $[e_i, e_i']=1$. 

If $\langle N_{\{l\}}^{(i)}, e_i, e_i' \rangle \slash N_{\{l\}}^{(i)}$ is not isomorphic to $\mathbb Z^2$, then there exists $x\in \langle e_i, e_i'\rangle$ such that $x\in N_{\{l\}}^{(i)}$ and $[x,e_i]=1$.

Otherwise, if $\langle N_{\{l\}}^{(i)}, e_i, e_i' \rangle \slash N_{\{l\}}^{(i)} \cong \mathbb Z^2$, then it follows that $\langle N_{\{l\}}^{(i)}, e_i, e_i' \rangle$ has finite index in $G_i$. Let $A$ be a free abelian vertex group such that $\langle e_i \rangle <A$ and let $a$ be an element in $A$ with the property that $a \notin \langle e_i, e'_i \rangle$. Then, since $\langle N_{\{l\}}^{(i)}, e_i, e_i' \rangle$ has finite index in $G_i$, there is $x\in \langle a,e_i, e_i \rangle$ (and so $[x, e_i]=1$) such that $x\in N_{\{l\}}^{(i)}$.

In conclusion, we have proved that there is $x\in N_{\{l\}}^{(i)}$ such that $[x, e_i]=1$.

Since $N_{\{l\}}^{(i)} \cap N_{\{k\}}^{(i)}$ has finite index in $N_{\{l\}}^{(i)}$, taking a power of $x$ (if neccesary) we may assume that $x\in N_{\{l\}}^{(i)} \cap N_{\{k\}}^{(i)}$. By the definitions this implies that there are elements
\[ (z_1,\dots, z_n)\in S \quad \text{and} \quad (y_1,\dots,y_n)\in S \]
where $z_i=x$, $z_l=1$, $y_i=x$ and $y_k=1$.

Recall that $\hat{s}_i$ equals $(x_1,\dots, x_{i-1}, s_i, x_{i+1},\dots, x_n)\in S$. We claim that $[z_k, x_k]\ne 1$ and $[y_l, x_l]\ne 1$.

Indeed, suppose by contradiction that $[z_k, x_k]=1$. Note that $z_k \in N_{\{i\}}^{(k)} \subseteq \langle N_{\{i\}}^{(k)}, x_k \rangle$. Since $\langle K_i, \hat{s}_i \rangle$ is finitely generated, so is $\langle N_{\{i\}}^{(k)}, x_k \rangle$. The group $\langle N_{\{i\}}^{(k)}, x_k \rangle$ is free (see Lemma \ref{lem:G/L is virtually abelian}), so if $[x_k,z_k]=1$, then we have that $z_k= x_k^n$ for some $n\in \mathbb{Z}$. But if $x_k^n= z_k \in N_{\{i\}}^{(k)}$ and $\langle N_{\{i\}}^{(k)}, x_k \rangle$ is finitely generated, then $N_{\{i\}}^{(k)}$ is of finite index in $\langle N_{\{i\}}^{(k)}, x_k \rangle$ and so it is also finitely generated. This is a contradiction since
\[ G_i \slash N_{\{k\}}^{(i)} \cong  G_k \slash N_{\{i\}}^{(k)},\]
which is isomorphic to $\mathbb Z^2$ by our standing assumption. Nevertheless, by Proposition \ref{Proposition 0} finitely generated normal subgroups are co-(virtually cyclic). Therefore, we have shown that $[z_k, x_k]\neq 1$. The argument to check that $[y_l, x_l]\neq 1$ is symmetric.

Now for each $m\in \mathbb{Z}$ we conjugate $(z_1,\dots, z_n)$ by $(x_1,\dots, x_{i-1}, s_i, x_{i+1},\dots, x_n)^m$ to get
\[ \mathbf z=(z_1^{x_1^m},\dots, z_n^{x_n^m})\in S,\]
and here $z_i^{x_i^m}=x$ and $z_l^{x_l^m}=1$. In the same way, for each $f\in \mathbb{Z}$ we have
\[\mathbf y= (y_1^{x_1^f},\dots, y_n^{x_n^f})\in S,\]
with $y_i^{x_i^f}=x$ and $y_k^{x_k^f}=1$. Then, since the tuples $\mathbf z$ and $\mathbf y$ have the same $i$-coordinate, we have that for all $m,l\in \mathbb{Z}$,
\[\mathbf y \mathbf z^{-1}=(y_1^{x_1^f}z_1^{x_1^{-m}},\dots, y_n^{x_n^f}z_n^{x_n^{-m}})\in K_i.\]
But the subspace generated by
\[ \Bigl\{(y_1^{x_1^f}z_1^{x_1^{-m}},\dots, y_n^{x_n^f}z_n^{x_n^{-m}}) \mid f,m\in \mathbb{Z} \Bigr\} \]
is infinite dimensional in $H_1(K_i; \mathbb{Q})$ and this is a contradiction because $H_0(\langle (x_1,\dots, x_{i-1}, s_i, x_{i+1},\dots, x_n) \rangle; H_1(K_i; \mathbb{Q}))$ is finite dimensional (see Remark \ref{remark:SES}). 

As a summary, $G_i \slash U_{\{l,k\}}(1)^{(i)}$ cannot be virtually $\mathbb{Z}^2$ and this proves the first statement of the Lemma.

We next address the second statement. Suppose that $G_i \slash U_J(1)^{(i)}$ is virtually $\mathbb{Z}$ and that $ U_J(1)^{(i)}$ is not finitely generated. 

From Lemma \ref{lemma:edgegroups} it follows that there is an edge group $\langle e_i \rangle$ in $G_i$ such that $\langle e_i \rangle \cap U_J(1)^{(i)} \ne 1$ and so $e_i^m\in  U_J(1)^{(i)}$ for some $m\in \mathbb{N}$. Without loss of generality we may assume that $m$ equals $1$. Since $e_i$ is an element in $U_J(1)^{(i)}$, by definition there are elements in $S$ of the form $(x_1^j,\dots, x_n^j)$ with $x_j^j=1$ and $\prod_{j\in J} x_i^j=e_i$. Moreover, from Lemma \ref{claim:Ki} we have that there is $s_i \in \langle e_i \rangle$ such that $\langle K_i, (x_1,\dots, x_{i-1}, s_i, x_{i+1},\dots, x_n)\rangle$ is finitely generated, where $(x_1,\dots, x_{i-1}, s_i, x_{i+1},\dots, x_n)$ equals $\prod_{j\in J} (x_1^j,\dots, x_n^j)$. Hence, if $k\in J$, then $x_k \in U_{J\setminus \{k\}}(1)^{(k)}$.

By taking the images of $\langle K_i, (x_1,\dots, x_{i-1}, s_i, x_{i+1},\dots, x_n)\rangle$ we get that $\langle N_{\{i\}}^{(k)},x_k \rangle$ is finitely generated, and it is also contained in $U_{J_k}(1)^{(k)}$. We then have the chain
\[ N_{\{i\}}^{(k)} < \langle N_{\{i\}}^{(k)},x_k \rangle < U_{J_k}(1)^{(k)} < G_k, \]
and by Corollary \ref{lem:prop_1_subdirect} $G_k \slash  N_{\{i\}}^{(k)}$ is virtually $\mathbb{Z}^m$ for some $m\in \{0,1,2\}$. Therefore, there are three options:
\begin{itemize}
    \item[(1)] The group $N_{\{i\}}^{(k)}$ has finite index in $\langle N_{\{i\}}^{(k)},x_k \rangle$, so in particular it is finitely generated. Thus, $G_k \slash N_{\{i\}}^{(k)}$ is either finite or virtually $\mathbb{Z}$ (see Propositon \ref{Proposition 0}), and in both cases we get that $U_{J_k}(1)^{(k)}$ is finitely generated.
    \item[(2)] The group $U_{J_k}(1)^{(k)}$ has finite index in $G_k$, so in particular it is finitely generated.
    \item[(3)] The group $\langle N_{\{i\}}^{(k)},x_k \rangle$ has finite index in $U_{J_k}(1)^{(k)}$, so $U_{J_k}(1)^{(k)}$ is finitely generated.
\end{itemize}
In conclusion, in the three possibilities we get that $U_{J_k}(1)^{(k)}$ is finitely generated, as we wanted.
\end{proof}

We note the following observation that will make the proof of the second property from Theorem \ref{thm:fp_subdirect_nilpotent quotient} easier.

\begin{remark}\label{remark:easier}
Let $G\coloneqq G_1 \times \dots \times G_n$ where $G_i\in \mathcal G$ is finitely generated and let $S<G$ be a finitely presented full subdirect product. Then, in order to show condition (2) of Theorem \ref{thm:fp_subdirect_nilpotent quotient}, from Lemma \ref{lemma:UJfg} and Lemma \ref{lem:justfg} we deduce that it is enough to check it when $U_J(1)^{(i)}$ is finitely generated.  
\end{remark}

\begin{lemma}[Complements of normal subgroups of type $FP_n$]\label{lem:edges_finite_index}
Let $G\coloneqq G_1 \times \dots \times G_n$ where $G_i\in \mathcal G$ is finitely generated and let $M$ be a non-trivial normal subgroup of $G$ of type $FP_n$. 
    
Then there is a finite index subgroup $H < G$ that satisfies that for each edge group $\langle e_i \rangle$  of $G_i$ in its standard splitting, there is $r_i\in \mathbb N$ such that $H= M\langle e_1^{r_1}, \dots, e_n^{r_n} \rangle$.
\end{lemma}
   
\begin{proof}
Since $M$ is of type $FP_n$, it is in particular finitely generated, so by Lemma \ref{lem:fg_normal} we have that $G\slash M$ is virtually abelian. Hence, there is $\tilde{G}$ a finite index subgroup in $G$ where $\tilde{G} \slash M$ is abelian. Let us define $H$ to be $(\tilde{G} \cap G_1) \times \dots \times (\tilde{G} \cap G_n)$. Then $H$ is a finite index subgroup of $G$ and $MH \slash M$ is abelian, so $H \slash (H\cap M)$ is also abelian. Since $H \cap M$ is of finite index in $M$, it is also of type $FP_n$. In conclusion, by renaming the groups we can assume that $M$ is normal in $G$ and $G\slash M$ is abelian.

Consider now $\chi \colon G \to \mathbb{Z}^k$ be the epimorphism such that $\ker \chi=M$, and for each edge group $\langle e_i \rangle <G_i$, define $E\coloneqq E(e_1, \dots, e_n)$ to be the free abelian group $\langle e_1, \dots, e_n \rangle < G$.

Let us firstly show that $\chi(E)$ has finite index in $\mathbb{Z}^k$.
Indeed, suppose towards contradiction that $\chi(E)$ does not have finite index in $\mathbb Z^k$, so that $\chi(E)\cong \mathbb Z^l$ for some $l<k$. Then there is a non-zero homomorphism $\phi \colon \mathbb{Z}^k \to \mathbb{Z}$ such that $(\phi \circ \chi)(E)=0$ and so $(\phi \circ \chi)(e_i)=0$  for all $i\in \{1,\dots,n\}$. In addition, since $\chi$ is an epimorphism, we can ensure that $\phi \circ \chi \neq 0$. By assumption, $ker \chi$ is of type $FP_n$, so $\ker (\phi \circ \chi)$ is also of type $FP_n$, and so $[\phi \circ \chi]\in \Sigma^n(G; \mathbb{Z})$. Thus, $[(\phi \circ \chi)_{i}]\in \Sigma^1(G_i)$ for some $i\in \{1,\dots,n\}$, where $(\phi \circ \chi)_{i}$ denotes the restriction of $\phi \circ \chi$ to $G_i$. However, we have that $(\phi \circ \chi)_{i}(e_i)=0$, that is, the character $(\phi \circ \chi)_{i}$ is trivial in an edge group of $G_i$, and this is a contradiction by \cite[Proposition 2.5]{CashenLevitt}. Hence, we have proved that $\chi(E)$ has finite index in $G\slash M$ and so the preimage $ME$ has finite index in $G$.

Given $e_1, \dots, e_{i-1}, e_{i+1}, \dots, e_n$, let us define $\tilde{H}_i$ to be
\[\bigcap_{e_{ij}\in E(G_i)} M E(e_1, \dots,e_{ij}, \dots, e_n),\]

where $E(G_i)$ denotes the finite set of edge groups in the decomposition of $G_i$ and $e_{ij} \in E(G_i)$ is the generator of the corresponding edge group.

Note that $\tilde{H}_i$ has finite index in $G$ for being a finite intersection of finite index subgroups and so it also has finite index in $M E(e_1, \dots, e_n)$ for all $e_i\in E(G_i)$. Since $M \langle e_1, \dots, e_{i-1}, e_{i+1}, \dots, e_n\rangle < \tilde H_i$, we have that for each $e_{ij}\in E(G_i)$ there exists $r_{ij}\in \mathbb N$ such that $\tilde H_i = M \langle e_1, \dots, e_{ij}^{r_{ij}}, \dots, e_n\rangle$.

Finally, if we define $\tilde H$ to be $\bigcap_{1\le i\le n} \tilde{H}_i$, then for each $e_i \in E(G_i)$ and $i\in \{1, \dots, n\}$, there exists $r_i\in \mathbb N$ such that $\tilde H = M\langle e_1^{r_1}, \dots, e_n^{r_n}\rangle$.
\end{proof}

\begin{lemma}[Commutators of elliptic elements]\label{lem:elliptic_commute}
Let $G\coloneqq G_1 \times \dots \times G_n$ where $G_i\in \mathcal G$ is finitely generated and let $M$ be a non-trivial normal subgroup of $G$ of type $FP_n$. Let $H$ be the finite index subgroup given in Lemma \ref{lem:edges_finite_index}.

Then the group $H_1 \times \dots \times H_n$, where $H_i = H \cap G_i$, is a finite index subgroup that satisfies that for each pair of elliptic elements $x,x'\in H_i$ we have that $[x,x'] \in [M,M]$.
\end{lemma}

\begin{proof}
We split this proof in several claims to make it more readable. Let $e_{i1}, \dots, e_{it}$ be generators of the edge groups in $G_i$ such that $[e_{ij}, e_{i,j+1}]=1$ for every $j\in \{1, \dots, t-1\}$. From Lemma \ref{lem:edges_finite_index}, for each generator $e_k$ of an edge group of $G_k$, where $1\le k\ne i \le n$, and each $j\in \{2, \dots, t\}$, there exist $r_{k}, r_{i1}, r_k', r_{i_2} \in \mathbb N \cup \{0\}$ such that 

\begin{equation}\label{eq:dec_H}
H=M\langle e_1^{r_1}, \dots, e_{i1}^{r_{i1}}, \dots, e_n^{r_n}\rangle = M \langle e_1^{r_1'}, \dots, e_{ij}^{r_{i2}}, \dots, e_n^{r_n'}\rangle    
\end{equation} 

has finite index in $G$.

Let $x$ be an elliptic element in $H_i$, namely $x$ belongs to a vertex group $G_v$ in the decomposition of $H_i$ and let $ y \in \langle e_k \mid 1\le k\le n, i\ne k \rangle$. By definition, $[H_i, y
]=1$. Suppose that $e_{is} \in G_v$ for $1\le s \le t$.

We first prove this particular instance of the statement: 

\begin{claim}\label{claim1inproof}
If $xy\in M$, then $[xy, e_{i1}^{r_{i1}}]\in [M,M]$.    
\end{claim}

\begin{proof}[Proof of the claim]

Let us show it by induction on $s$.

If $s=1$, then $e_i=e_{i1}\in G_v$, and since the vertex groups are abelian we have that $[x,e_i]=1$. Furthermore, the element $y$ lies in $\langle e_k \mid 1\le k \le n, i\ne k 
\rangle$, so we have that $[y,e_i]=1$. From the commutator identities, we then obtain that $[xy,e_i]=1$ and so $[xy, e_i^{r_i}]=1$.

Let us now show the induction step. Assume that $[xy, e_{ij}^{r_{ij}}]\in [M,M]$ for $j\in \{2, \dots, s\}$ and we need to show that $[xy, e_{i1}^{r_{i1}}]\in [M,M]$. From the two decompositions of $H$ given in \eqref{eq:dec_H}, we have that $e_{i_1}^{r_{i1}} \in M \langle e_1^{r_1'}, \dots, e_{i2}^{r_{i2}}, \dots, e_n^{r_n'}\rangle$ and so there exist $l_{2}\in \mathbb N\cup \{0\}$ and $w\in \langle e_1^{r_1}, \dots, e_{i-1}^{r_{i-1}}, e_{i+1}^{r_{i+1}}, \dots, e_n^{r_n} \rangle$ such that $e_{i1}^{r_{i1}} e_{i2}^{-l_2r_{i2}} w^{-1} \in M$.

Notice that
$$[xy, e_{i1}^{r_{i1}}] = [xy,  (e_{i1}^{r_{i1}} e_{i2}^{-l_2r_{i2}} w^{-1}) (w e_{i2}^{l_2r_{i2}})]=[xy, (w e_{i2}^{l_2r_{i2}})] [xy,  (e_{i1}^{r_{i1}} e_{i2}^{-l_2r_{i2}} w^{-1})]^g,
$$
where the second equality is derived from commutator identities and $g$ equals $w e_{i2}^{l_2r_{i2}}$.

The elements $y$ and  $w$ belong to the abelian subgroup $\langle e_1, \dots, e_{i-1}, e_{i+1}, \dots, e_n \rangle$, so $[y,w]=1$. Furthermore, since $x\in G_i$, we have that $[x,w]=1$. By induction, $[xy,e_{i2}^{r_{i2}}]$ lies in $[M,M] $ and by using commutator identities we get that $[xy, (w e_{i2}^{l_2r_{i2}})]\in [M,M]$. Moreover, $(e_{i1}^{r_{i1}} e_{i2}^{-l_2r_{i2}} w^{-1}) \in M$, so $[xy, (e_{i1}^{r_{i1}} e_{i2}^{-l_2r_{i2}} w^{-1})]\in [M,M]$. Combining the above, we conclude that $[xy, e_{i1}^{r_{i1}}]\in [M,M]$.
\end{proof}

\begin{claim}\label{claim2inproof}
For each pair of generators $e_{i1}$ and $e_{is}$ of edge groups of $G_i$, we have that $[e_{i1}^{r_{i1}}, e_{is}^{r_{is}}]\in [M,M]$.     
\end{claim}

\begin{proof}
The proof is again by induction on $s$.

If $s=1$, then it is obvious. Suppose by induction that $[e_{ij}^{r_{ij}}, e_{is}^{r_{is}}]\in [M,M]$ for $j\in \{2, \dots, s\}$ and let us show that $[e_{i1}^{r_{i1}}, e_{is}^{r_{is}}]\in [M,M]$.

As argued above, from the decompositions of $H$ given in \eqref{eq:dec_H}, there are $l_2 \in \mathbb N\cup \{0\}$ and an element $w\in \langle e_1^{r_1}, \dots, e_{i-1}^{r_{i-1}}, e_{i+1}^{r_{i+1}}, \dots, e_n^{r_n} \rangle$ such that $e_{i1}^{r_{i1}}e_{i2}^{-l_{2}r_{i2}} w^{-1} \in M$ and by inductive hypothesis, $[e_{i2}^{r_{i2}}, e_{is}^{r_{is}}]$ lies in $[M,M]$. We want to show that $[e_{i1}^{r_{i1}},e_{is}^{r_{is}}] \in [M,M]$. We have that
$$
[e_{i1}^{r_{i1}},e_{is}^{r_{is}}]=[(e_{i1}^{r_{i1}}e_{i2}^{-l_{2}r_{i2}} w^{-1}) (w e_{i2}^{l_2r_{i2}}), e_{is}^{r_{is}}]=[(e_{i1}^{r_{i1}}e_{i2}^{-l_{2}r_{i2}} w^{-1}), e_{is}^{r_{is}}]^g [w e_{i2}^{l_2r_{i2}}, e_{is}^{r_{is}}]
$$
where the second equality is obtained using commutator identities. 

The elements $e_{i1}, e_{12}$ belong to a vertex group, say $G_v$, and so $x\coloneqq e_{i1}^{r_{i1}}e_{i2}^{-l_{2}r_{i2}} \in G_v$. Furthermore, $y\coloneqq w^{-1} \in \langle e_1, \dots, e_{i-1}, e_{i+1}, \dots, e_n \rangle$ and $xy \in M$. Then, from the above claim, we get that \[[(e_{i1}^{r_{i1}}e_{i2}^{-l_{2}r_{i2}} w^{-1}), e_{is}^{r_{is}}]=[xy, e_{is}^{r_{is}}]\in [M,M].\] Moreover, $w$ and $e_{is}$ commute, so $[(w e_{i2}^{l_2r_{i2}}), e_{is}^{r_{is}}] = [e_{i2}^{l_{2}r_{i2}},e_{is}^{r_{is}}]\in [M,M]$. Hence, we have shown that $[e_{i1}^{r_{i1}},e_{is}^{r_{is}}] \in [M,M]$.
\end{proof}

Finally, let $x, x' \in H \cap G_i$ be two elliptic elements. We claim that $[x,x']\in [M,M]$.

Suppose that $x\in G_v$ and $x'\in G_w$, where $G_v$ and $G_w$ are vertex groups in $G_i$. Let $e_{i1}\in G_v$ and $e_{is}\in G_w$ be two elements in the edge groups. From the decompositions of $H$ given in \eqref{eq:dec_H}, we have that $x e_{i1}^{r_{i1}l_{x}} w_x \in M$ and $x' e_{is}^{r_{is}l_{x'}} w_{x'}\in M$ for some $l_x,l_{x'}\in \mathbb Z$, $w_x \in \langle e_1^{r_1}, \dots, e_{i-1}^{r_{i-1}}, e_{i+1}^{r_{i+1}}, \dots, e_n^{r_n}\rangle$ and $w_{x'} \in \langle e_1^{r_1'}, \dots, e_{i-1}^{r_{i-1}'}, e_{i+1}^{r_{i+1}'}, \dots, e_n^{r_n'}\rangle$.

Using commutator identities we have
$$
[x,x']=[(x e_{i1}^{r_{i1}l_{x}} w_x) (w_x^{-1}e_{i1}^{r_{i1} l_x}), (x' e_{is}^{r_{is}l_{x'}} w_{x'}) (w_{x'}^{-1}e_{is}^{r_{is} l_{x'}})]=
$$

$$[(x e_{i1}^{r_{i1}l_{x}} w_x), (w_{x'}^{-1}e_{is}^{r_{is}l_{x'}})]^{g_1}[(x e_{i1}^{r_{i1}l_{x}} w_x),({x'} e_{is}^{r_{is}l_{{x'}}} w_{x'})]^{g_2}[(w_x^{-1}e_{i1}^{r_{i1}l_x}), (w_{x'}^{-1}e_{is}^{r_{is}l_{x'}})]^{g_3}[(w_x^{-1}e_{i1}^{r_{i1}l_x}), ({x'} e_{is}^{r_{is}l_{{x'}}} w_{x'})]^{g_4}.
$$

Using the fact that $w_x, w_{x'}$ belong to the abelian group $\langle e_1,\dots, e_n\rangle$, they commute and they also commute with $x,{x'}, e_{i1}$ and $e_{i2}$. Then we deduce that \[[(x e_{i1}^{r_{i1}l_{x}} w_x), (w_{x'}^{-1}e_{is}^{r_{is}l_{x'}})]=[(x e_{i1}^{r_{i1}l_{x}} w_x), (e_{is}^{r_{is}})] \in [M,M]\] by Claim \ref{claim1inproof}. Similarly, $[(w_x^{-1}e_{i1}^{r_{i1}l_x}), ({x'} e_{is}^{r_{is}l_{{x'}}} w_{x'})]\in [M,M]$. Furthermore, $[(w_x^{-1}e_{i1}^{r_{i1}}), (w_{x'}^{-1}e_{is}^{r_{is}l_{x'}})]=[e_{i1}^{r_{i1}}, e_{is}^{r_{is}l_{x'}}]$ lies in $[M,M]$ by Claim \ref{claim2inproof}. Finally, $[(x e_{i1}^{r_{i1}l_{x}} w_x),({x'} e_{is}^{r_{is}l_{{x'}}} w_{x'})]\in [M,M]$ since the elements $(x e_{i1}^{r_{i1}l_{x}} w_x), ({x'} e_{is}^{r_{is}l_{{x'}}} w_{x'})$ belong to $M$. Thus, we conclude that $[x,{x'}]\in [M,M]$.
\end{proof}

\begin{lemma}[Quotients by elliptic commutators]\label{lem:HNN_structure_quotient}
Let $G\coloneqq G_1 \times \dots \times G_n$ where $G_i\in \mathcal G$ is finitely generated and let $M$ be a non-trivial normal subgroup of $G$ of type $FP_n$.

For $i\in \{1,\dots,n\}$ let $H_i< G_i$ be the group given in Lemma \ref{lem:elliptic_commute} and let $H\coloneqq H_1\times \cdots \times H_n$. In particular, $H$ satisfies that for each pair of elliptic elements $x,y\in H_i$ we have that $[x,y] \in [M,M]$.

The subgroup $H_i \in \mathcal G$ admits a graph of groups decomposition $\langle T_i, t_1, \dots, t_m\rangle$ where $T_i$ is the normal subgroup generated by the elliptic elements of $H_i$ and $t_1,\dots,t_m$ are stable letters.

Then $H_i\slash (T_i \cap [M,M])$ is isomorphic to a graph of groups with a unique abelian vertex group.
\end{lemma}

\begin{proof}
As a graph of groups, $H_i$ has a presentation of the form
\[ \langle S_i, t_1,\dots,t_m \mid t_i^{-1}A_i t_i= B_i, i\in \{1,\dots,m\} \rangle, \]
where $S_i$ is an amalgamated free product, $t_1,\dots,t_m$ are the stable letters and $A_i, B_i$ are infinite cyclic edge groups. We may also consider the presentation of $H_i$ of the form 
\[ \Big \langle  T_i , t_1,\dots,t_m \mid t_i^{-1}A_i t_i= B_i, i\in \{1,\dots,m\} \Big \rangle, \]
where $T_i =\langle \langle S_i \rangle \rangle$ is the normal closure of $S_i$ in $H_i$ and so it is generated by elliptic elements. In particular, $H_i= T_i \rtimes F_m$ where $F_m$ is the free group of rank $m$ generated by the $t_j$. Note that $T_i \cap [M,M]$ is normal in $H_i$ and $T_i \cap [M,M]< T_i$. We also have that $H_i \slash (T_i\cap [M,M])$ is isomorphic to
\[ \overline{H}_{i}= \langle  T_i \slash ( T_i \cap [M,M] ), \tau_1, \dots, \tau_m \mid \tau_i^{-1} A_i \slash (A_i \cap [M,M]) \tau_i= B_i\slash (B_i \cap [M,M]), i\in \{1, \dots, m\} \rangle.\]
Indeed, there is an epimorphism $H_i=T \rtimes F_m \to \overline{H}_i = (T_i\slash (T_i \cap [M,M]))\rtimes F_m$ which induces an isomorphism on the free group $F_m$. It follows that the kernel is precisely $T_i\cap [M,M]$ and so by the First Isomorphism Theorem we obtain what we want. Moreover, from Lemma \ref{lem:elliptic_commute} we have that $T_i  \slash \big ( T_i \cap [M,M]\big )$ is abelian and so the statement follows.
\end{proof}

\begin{lemma}[Quotient by the commutator subgroup of a normal subgroup of type $FP_n$]\label{lem:nilpotent}
Let $G\coloneqq G_1 \times \dots \times G_n$ where $G_i\in \mathcal G$ is finitely generated and let $M$ be a non-trivial normal subgroup of $G$ of type $FP_n$. Then $G\slash [M,M]$ is virtually nilpotent of class at most 2.
\end{lemma}

\begin{proof}
Let $H= H_1 \times \dots \times H_m < G$ be the finite index subgroup of $G$ given in Lemma \ref{lem:elliptic_commute}.

Our goal is to show that $[H_i,[H_i,H_i]]\subseteq [M,M]$. In particular, we will get that $H \slash [M,M]$ is nilpotent of class at most $2$. Without loss of generality, we prove it for $H_1$.

Let $h_1, h_2, h_3$ be elements in  $H_1$. Since $H=M\langle e_1^{r_1}, \dots, e_n^{r_n}\rangle$ for some edge groups $e_i$ of $G_i$ (see Lemma \ref{lem:edges_finite_index}), we have that $h_i= m_i g_i$ for $m_i \in M$ and $g_i \in \langle e_1^{r_1},\dots, e_n^{r_n} \rangle$. Then, \[[h_2, h_3]= [m_2g_2, m_3g_3]= [m_2,m_3g_3]^{g_2}[g_2,m_3g_3]=[m_2,g_3]^{g_2}[m_2,m_3]^{g_3g_2}[g_2,g_3][g_2,m_3]^{g_3}.\]
Note that the elements $[m_2,g_3], [m_2,m_3], [g_2,m_3]$ lie in $M$, so then the commutator of $m_1$ with those three elements belongs to $[M,M].$ As a consequence, in order to check that $[m_1g_1, [m_2g_2, m_3g_3]]$ is an element of $[M,M]$, it suffices to show that $[H_1, \langle e_1^{r_1} \rangle], \langle e_1^{r_1} \rangle]\subseteq [M,M]$.

From Lemma \ref{lem:elliptic_commute} we deduce that we are left to check that if $t$ is a hyperbolic element of $H_1$, then $[[t, e_1^{r_1}], e_1^{r_1}]$ is in $[M,M]$. Suppose that the hyperbolic element $t$ gives the relation $t e t^{-1}= e'$, where $e$ and $e'$ are in edge groups of $H_1$. Again by Lemma \ref{lem:edges_finite_index} we get that $H=M \langle e^r, e_2^{r_2'}, \dots, e_n^{r_n'}\rangle$. Hence, \[e_1^{r_1} e^{rm_{1}} w\in M \quad \text{and} \quad  t e_1^{r_1s_{1}}\tilde{w}\in M,\] for some integers $m_1,s_1$ and some $w, \tilde{w}$ in the abelian group $\langle e_2, \dots, e_n
\rangle$.

Using commutator identities, we see that
\[[t e_1^{r_1s_{1}}\tilde{w}, e_1^{r_1}e^{rm_{1}} w]=[t,e_1^{r_1}]^{g_1}[ t,e^{rm_{1}}]^{g_2}[e_1^{r_1s_{1}},e^{rm_{1}}]^{g_3},\] for some elements $g_1,g_2,g_3$ in $H_1$. Lemma \ref{lem:elliptic_commute} tells us that $[e_1^{r_1s_{1}},e^{rm_{1}}] \in [M,M]$, so we deduce that $[[t, e_1^{r_1}], e_1^{r_1}]$ lies in $[M,M]$ if and only if $[[t,e^{rm_{1}}], e_1^{r_1}]\in [M,M]$.

Note that $[t, e^{rm_{1}}]$ equals $ t e^{rm_{1}} t^{-1} e^{-r}= {e'}^{rm_{1}} e^{-r},$ so
\[[[t,e^{rm_{1}}], e_1^{r_1}]= [{e'}^{rm_{1}} e^{-r}, e_1^{r_1}],\]
and once more from the commutator identities and Lemma \ref{lem:elliptic_commute} we conclude that \[[{e'}^{rm_{1}} e^r, e_1^{r_1}]\in [M,M].\] 
\end{proof}

\begin{lemma}[Commutator subgroups of a normal subgroup of type $FP_n$]\label{lem:normal_subnormal}
Let $G\coloneqq G_1 \times \dots \times G_n$ where $G_i\in \mathcal G$ is finitely generated and let $M$ be a non-trivial normal subgroup of $G$ of type $FP_n$. Then there is a finite index subgroup $H$ of $G$ such that $[H,H]=[M,M]$.
\end{lemma}

\begin{proof}

Let $H=H_1 \times \dots \times H_n$ be as in Lemma \ref{lem:elliptic_commute}. The first goal is to show that the generators of $[H,H]$ have finite order modulo $[M,M]$. We prove that claim for the generators of $[H_1,H_1].$

Lemma \ref{lem:HNN_structure_quotient} tells us that $H_1\slash (T_1 \cap [M,M])$ is isomorphic to the graph of groups
$$\overline{H}_1=\langle T_1\slash (T_1\cap [M,M]), \tau_1, \dots, \tau_m \mid {A_i\slash (A_i\cap [M,M])}^{\tau_i} = B_i\slash (B_i \cap [M,M]), i\in \{1, \dots, m\}\rangle.$$
Furthermore, $H_1\slash (T_1 \cap [M,M]) < H\slash [M,M]$ is nilpotent because by Lemma \ref{lem:nilpotent} $H \slash [M,M]$ is nilpotent. Therefore, we have that $\overline{H}_1$ is nilpotent and so, in particular, it is residually finite.

Hence, by \cite[Theorem 3.1]{Kim} there is a non-trivial subgroup of $A_i \slash ( A_i \cap [M,M])$ which is normal in the group $\overline{H}_1 \slash ( T_1 \cap [M,M])$. Then, if $A_i \slash ( A_i \cap [M,M])= \langle \bar{e_{1i}} \rangle$, there is $k_j \in \mathbb{Z}$ such that $\tau_j \bar{e_{1i}}^{k_j} \tau_j= \bar{e_{1i}}^{k_j}$ or $\tau_j \bar{e_{1i}}^{k_j} \tau_j= \bar{e_{1i}}^{-k_j}$ for each $j\in \{1,\dots,m\}$. Therefore, by taking $k$ to be $\max \{k_j\}$, we have that $[t_j^2, e_{1i}^{k}]\in [M,M]$ for $j\in \{1, \dots, m\}$.

Taking a finite index subgroup of $H$ if necessary, we may assume that $H=M\langle e_{1i}^{r_1k}, \dots, e_n^{r_m}\rangle$. Let $h_1, h_2$ be two standard generators of $H_1$. Then there are $m_i\in \mathbb{Z}$ and $w_i \in \langle e_2^{r_2}, \dots, e_m^{r_m}\rangle$ such that $h_i^2 (e_1^{r_1k})^{m_i}w_i \in M$. Therefore, using commutator identities we have that the commutator $[h_1^2 (e_1^{r_1k})^{m_1}w_1, h_2^2 (e_1^{r_1k})^{m_2}w_2]=[h_1^2, h_2^2]^g$ belongs to $[M,M]$.

Since by Lemma \ref{lem:nilpotent} we have that $H \slash [M,M]$ is nilpotent of class at most $2$, and $[h_1^2, h_2^2]\in [H,H]$ (which is central in the nilpotent group of class 2) we deduce that $[h_1^2, h_2^2]=[h_1,h_2]^4$ lies in $[M,M]$. 

The subgroup $H$ is of finite index in $G$, so in particular, it is finitely generated. Then, the quotient group $[H,H]\slash [[H,H],H]$ is also finitely generated. Since the group $[H,H]\slash [M,M]$ is abelian and each generator $[h_1,h_2] $ in $[H,H]$ satisfies that $[h_1,h_2]^4\in [M,M]$, we conclude that $[H,H]\slash [M,M]$ is finite. Therefore, $H\slash [M,M]$ is finite-by-abelian and as a consequence, also virtually abelian. Hence there is a finite index subgroup $J$ in $H$ such that $[M,M]<J$ and $[J,J]=[M,M]$.
\end{proof}

We now have all the tools to prove condition (2) from Theorem \ref{thm:fp_subdirect_nilpotent quotient}.

\begin{lemma}[Condition (2)]\label{lem:prop_2_subdirect}
Let $G\coloneqq G_1 \times \dots \times G_n$ where $G_i\in \mathcal G$ is finitely generated. Let $S<G$ be a finitely presented full subdirect product. Then for each $J\subseteq I\setminus \{i\}$ such that $U_J(1)^{(i)}$ is finitely generated and for each finite index subgroup $K < U_J(1)^{(i)}$, there is a finite index subgroup $H_K$ of $G_i$ such that $K< H_K$ and $[H_K,H_K]=[K,K]$.
\end{lemma}

\begin{proof}
Let $J \subseteq I \setminus \{i\}$ and $K$ be as in the statement. Since $U_J(1)^{(i)}$ is finitely generated and $K$ is of finite index in this group, we may assume that $K$ is characteristic in $U_J(1)^{(i)}$, and as $U_J(1)^{(i)}$ is normal in $G_i$, we conclude that $K$ is normal in $G_i$. The subgroup $K$ is in addition finitely generated for being a finite index subgroup of $U_J(1)^{(i)}$, so applying Lemma \ref{lem:normal_subnormal} for $n=1$ we get that there is $H_K$ a finite index subgroup of $G_i$ with the conditions of the statement.
\end{proof}

We summarize everything proved in one main statement.

\begin{theorem}[Structure of finitely presented full subdirect products in $\mathcal{G}$]\label{thm:main}
Let $G\coloneqq G_1 \times \dots \times G_n$ where $G_i\in \mathcal G$ is finitely generated and let $S<G$ be a finitely presented full subdirect product. Define $L$ to be $L_1\times \cdots \times L_n$, where $L_i= S \cap G_i$.

Then $L < S < G$ and $G\slash L$ is virtually nilpotent. In particular, there is a subnormal series $$S_0=M_0 \lhd M_1 \lhd \dots \lhd M_{k-1} \lhd M_k$$ where $S_0$ has finite index in $S$, $M_k$ has finite index in $G$ and $M_i\slash M_{i-1}$ is maximal abelian (i.e. if $M_{i-1} \lhd M_i'$, $M_i< M_i'$ and $M_i'\slash M_{i-1}$ is abelian, then $M_i=M_i'$).
\end{theorem}

\begin{proof}
Let us first conclude that for each $i\in \{1,\dots,n\}$ the quotient $G_i \slash L_i$ is virtually nilpotent. For that, it suffices to check that conditions (1) and (2) from Theorem \ref{thm:fp_subdirect_nilpotent quotient} hold.

The property (1) is shown in Corollary \ref{lem:prop_1_subdirect}. For property (2), as mentioned in Remark \ref{remark:easier} it suffices to check it for the case where $U_J(1)^{(i)}$ is finitely generated, and this is precisely proved in Lemma \ref{lem:prop_2_subdirect}.

Hence, since $G_i \slash L_i$ is virtually nilpotent for all $i\in \{1,\dots,n\}$, the quotient $G\slash L$ is also virtually nilpotent. Moreover, there are $H_i < G_i$ and $\tilde{S}<S$ of finite index such that
\[ L < \tilde{S} < H=H_1\times \cdots \times H_n\]
and $H\slash L$ is nilpotent. The subgroup $\tilde{S}\slash L$ is therefore subnormal, so there is a subnormal series as in the statement.
\end{proof}

\section{Subgroup of type $FP_n$}\label{Section3}

The goal of this section is to study subgroups of type $FP_n$ and prove the following result:

\begin{theorem}[Structure of full subdirect products of type $FP_n$]\label{thm:main2}
Let $G\coloneqq G_1 \times \dots \times G_n$ where $G_i\in \mathcal G$ is finitely generated and let $S<G$ be a full subdirect product of type $FP_n$.

There there is $S_0$ a finite index subgroup of $S$ and $G_0$ a finite index subgroup of $G$ such that $S_0$ is normal in $G_0$ and $G_0 \slash S_0$ is free abelian.
\end{theorem}

\begin{proof}
From Theorem \ref{thm:main}, we have that if $G_1,\dots,G_n$ are finitely generated and belong to $\mathcal{G}$, and $S$ is a finitely presented full subdirect product of $G_1\times \cdots \times G_n$, then there is a subnormal series
\[ S_0 \triangleleft S_1 \triangleleft \cdots \triangleleft S_c \]
such that $S_{i+1}\slash S_i$ is maximal abelian, $S_0$ has finite index in $S$ and $S_c$ has finite index in $G_1\times \cdots \times G_n$. Hence, up to finite index, we may assume that the quotient $\frac{S_c}{S_{c-1}}$ is $\mathbb{Z}^k$ for some $k\in \mathbb{N}$.

Now suppose that $S$ is of type $FP_n$. Then, $S_i$ is also of type $FP_n$ for all $i\in \{0,\dots,c\}$. From Lemma \ref{lem:normal_subnormal} we deduce that passing again to a finite index subgroup (if necessary), we have that
\[ [S_c,S_c]=[S_{c-1},S_{c-1}]\subseteq S_{c-2}. \]
But by maximality, we get that the subnormal series has just one term, that is, $S_0$ is normal in $S_c$, and the quotient
\[ \frac{S_c}{S_0}\]
is abelian.
\end{proof}

\section{Algorithmic problems}\label{Section4}

We finish these notes by studying some algorithmic problems, namely the multiple conjugacy problem and the membership problem, for finitely presented full subdirect products of direct products of $2$-dimensional coherent RAAGs.

The \emph{multiple conjugacy problem} for a finitely generated group $G$ (given by a finite generating set) asks if there is an algorithm that, given a natural number $l$ and two $l$-tuples of elements in the generators of $G$, say $x=(x_1,\dots,x_l)$ and $y=(y_1,\dots,y_l)$, can determine if there exists $g\in G$ such that $g^{-1}x_i g= y_i$ in $G$ for $i\in \{1,\dots,l\}$.

If $G$ is a finitely generated group (given by a finite generating set) and $H$ is a finitely generated subgroup of $G$ (given by a finite set of words in the generators of $G$), the \emph{membership problem} asks if there is an algorithm that decides whether or not a given element $g$ in $G$ as a word in the generators belongs to $H$.

Suppose that $G_1,\dots, G_n$ are $2$-dimensional coherent RAAGs and let $S$ be a finitely presented full subdirect product of $G_1\times \cdots \times G_n$. We aim to show that the multiple conjugacy problem is decidable for $S$ and that the membership problem is decidable for finitely presented subgroups of $S$. The case $n=2$ is covered in \cite[Section 6]{CasalsLopezdeGamiz} and since the argument for the general case is identical, we just sketch the steps here and refer the reader to \cite[Section 6]{CasalsLopezdeGamiz}.

We recall two results from \cite{BridsonHowieMillerShort} that we need:

\begin{proposition}\cite[Proposition 7.1]{BridsonHowieMillerShort}\label{PropositionConjugacy}
Let $\Gamma$ be a bicombable group, let $H < \Gamma$ be a subgroup, and suppose that there exists a subgroup $L<H$ normal in $\Gamma$ such that $\Gamma \slash L$ is nilpotent. Then $H$ has a decidable multiple conjugacy problem.   
\end{proposition}

\begin{lemma}\cite[Lemma 7.2]{BridsonHowieMillerShort}\label{LemmaUnique}
Suppose $G$ is a group in which roots are unique and $H<G$ is a subgroup of finite index. If the multiple conjugacy problem for $H$ is decidable, then the multiple conjugacy problem for $G$ is decidable.    
\end{lemma}

Let us first check the multiple conjugacy problem.

\begin{theorem}[Decidability of the multiple conjugacy problem]
Suppose that $G_1,\dots, G_n$ are $2$-dimensional coherent RAAGs and let $S$ be a finitely presented full subdirect product of $G_1\times \cdots \times G_n$.

Then $S$ has decidable multiple conjugacy problem.    
\end{theorem}

\begin{proof}
The groups $G_i$ are CAT(0)  (see \cite[Corollary 9.5]{CasalsDuncanKazachkov}) and have unique roots (see, for instance, \cite[Lemma 6.3]{Minasyan}). By Theorem \ref{thm:main} the quotient $(G_1 \times \cdots \times G_n) \slash (L_1 \times \cdots \times L_n)$ is virtually nilpotent, so there is $G$ a finite index subgroup in $G_1\times \cdots \times G_n$ such that $L_1 \times \cdots \times L_n < G$ and $G \slash (L_1 \times \cdots \times L_n)$ is nilpotent. The group $G$ is CAT(0) for being a finite index subgroup of a CAT(0) group and so it is bicombable. Then, since $L_1 \times \cdots \times L_n < S \cap G < G$ and $(S \cap G) \slash (L_1 \times \cdots \times L_n)$ is nilpotent, it follows from Proposition \ref{PropositionConjugacy} that $G \cap S$ has decidable multiple conjugacy problem. The group $S$ has unique roots and $G \cap S$ has finite index in $S$, so from Lemma \ref{LemmaUnique} we conclude that $S$ has decidable multiple conjugacy problem.    
\end{proof}

We finish the section with the membership problem for finitely presented subgroups of $S$.

\begin{theorem}[Decidability of the membership problem]
Suppose that $G_1,\dots, G_n$ are $2$-dimensional coherent RAAGs and let $S$ be a finitely presented full subdirect product of $G_1\times \cdots \times G_n$.

Given a finitely presented subgroup $H$ of $S$, then the membership problem for $H$ in $S$ is decidable.    
\end{theorem}

\begin{proof}
The problem reduces to the decidability of the membership problem for $H< G_1\times \cdots \times G_n$. Let $g=(g_1,\dots,g_n)$ be an element in $G_1\times \cdots \times G_n$. We can determine algorithmically if $g_i \in H_i = \pi_i(H)$, where $\pi_i$ is the natural projection homomorphism $G_1\times \cdots \times G_n \to G_i$. If $g_i \notin H_i$, then $g\notin H$.

Otherwise, we replace $G_1\times \cdots \times G_n$ by $H_1\times \cdots \times H_n$. By Theorem \ref{thm:main}, the quotient group $Q= (H_1\times \cdots \times H_n)\slash (L_1\times \cdots \times L_n)$ is virtually nilpotent. Let $\varphi \colon H_1 \times \cdots \times H_n \to Q$ be the quotient map. Virtually nilpotent groups are subgroup separable, so if $\varphi(g)\notin \varphi(H)$, then there is a finite quotient of $Q$ that separates $g$ from $H$. But since $\ker \phi= L_1 \times \cdots \times L_n \subseteq H$, then $\varphi(g)\in \varphi(H)$ if and only if $g\in H$. If $\varphi(g)\notin \varphi(H)$, then an enumeration of finite quotients of $H_1\times \cdots \times H_n$ provides an effective procedure. We run this in parallel with an enumeration of $g^{-1}w$ for words $w$ in the generators of $H$ that will terminate if $g\in H$.
\end{proof}

\begin{remark}
For our purposes, we have stated the decidability problems for finitely presented full subdirect products of $2$-dimensional coherent RAAGs. However, we expect the results to hold for finitely presented full subdirect products of finitely generated groups in the class $\mathcal G$. In order to extend the results it suffices to show that groups in this class are $CAT(0)$ and have unique roots. In particular, this has already been proven for tubular groups whose underlying graph is a tree in \cite{Button} and so the results apply for groups in that class.
\end{remark}

\end{document}